\documentclass[12pt,twoside,reqno]{amsart}
\usepackage[colorlinks=true,citecolor=blue]{hyperref}
\usepackage{mathptmx, amsmath, amssymb, amsfonts, amsthm, mathptmx, enumerate, color,mathrsfs}
\usepackage{graphicx}

\usepackage{multirow}
\usepackage{epstopdf}
\usepackage{multicol}
\usepackage{algorithm}
\usepackage{algorithmic}
\usepackage{epstopdf}

\usepackage{cite}

\newtheorem{theorem}{Theorem}[section]

\newtheorem{proposition}[theorem]{Proposition}
\newtheorem{corollary}[theorem]{Corollary}

\theoremstyle{definition}
\newtheorem{definition}[theorem]{Definition}

\newtheorem{example}[theorem]{Example}

\newtheorem{remark}[theorem]{Remark}
\numberwithin{equation}{section}

\begin{document}
\setcounter{page}{1}

\vspace*{2.0cm}
\title[High-order tangent vectors]
{High order tangent vectors to sets with applications to constrained optimization problems}\footnote{In memory of Professor Rafail Gabasov}
\author[V.V. Gorokhovik]{ Valentin V. Gorokhovik}
\maketitle
\vspace*{-0.6cm}

\begin{center}
{\footnotesize

Institute of Mathematics, National Academy of Sciences of Belarus, Minsk, Belarus

}\end{center}

\vskip 4mm {\footnotesize \noindent {\bf Abstract.}
We introduce an extended tangent cone of high order to a set and study its properties. Then we use this local approximation for deriving high-order necessary conditions for local minimizers of constrained optimization problems.


}

\renewcommand{\thefootnote}{}
\footnotetext{ $^*$Corresponding author.
\par
E-mail address: gorokh@im.bas-net.by (V.V. Gorokhovik).
\par
Received xx, x, xxxx; Accepted xx, x, xxxx.

\rightline {\tiny   \copyright  2022 Communications in Optimization Theory}}

\section{Introduction}

Notions of first, second and higher order tangent vectors to sets are crucial in variational analysis. A far from complete but very numerous list of (mainly English-language) works devoted to this subject can be found in the review \cite{GIN}, the monograph \cite{KTZ}, and also in relatively recent papers \cite{Khanh_Tung18,Penot17,Tung}. We supplement this list with
works \cite{Bakh_Gor09a,Bakh_Gor09b,GK72,GK73,Gor83,Gor90,Gor05,Gor06a,Gor06b,Gor12,Gor-Rach,Gor_Sim01,Gor_Tr14} related to this area of research, written mainly in Russian and not included in the surveys of the publications mentioned above.

In this paper we focus on one of
possible definitions of high order tangent vectors to sets and consider some applications of this notion to constrained optimization problems. First of all we shortly discuss existing definitions of tangent vectors of first, second and higher order related to the notion that is introduced here.

There are a large number of different first-order tangent vectors to sets (in particular, such as the radial tangent vector, the feasible tangent vector, the contingent vector, the interiorly contingent vector, the adjacent vector, the interiorly adjacent vector) and their second-order counterparts; their definitions, characterizations, examples, and comparisons can be found in \cite{GIN}, \cite[Chapter 4]{KTZ}.
%and \cite{Palata}.
Here we deal only with first- and second-order contingent vectors as well as their extensions to high-order ones. Following  \cite{RW}, we call first-order contingent vectors simply by tangent ones.% and $k$th-order contingent vectors by $k$th-order proper tangent vectors.

%Let us recall the existing definitions of tangent (contingent)  vectors of the first and second orders.

Throughout the paper we use the following notations.
By $S$
we denote the family of all sequences of positive
real numbers $(t_n),\,\,t_n>0,\,\, n=1,2,\dots;$ $S(\alpha)$ is
the subfamily of $S,$ consisting of such sequences
$(t_n),$ that converge to a real number $\alpha \ge 0;$ and
$S(\infty)$ is the subfamily of sequences of positive
real numbers converging to $+\infty.$

Let $X$ be a real normed space,  $Q$  a nonempty subset of $X$.

By $U_Q$ we denote the family of all sequences $(x_n)$ such that $x_n \in Q, \,\, n=1,2,\dots;$ $U_Q(x)$ is the subfamily of $U_Q,$ consisting of such sequences $(x_n),$ which converge (by the norm of the space $X$) to $x;$ $U_Q(\infty)$ is the subfamily of sequences $(x_n)$ from $U_Q$ such that $\|x_n\| \to \infty$ as $n \to \infty.$

A vector $h \in X$ is said \cite{RW,Aub_Fr,Bon_Sh,Ioffe,Mord,Penot13} to be {\it a first-order tangent (contingent) vector
to a set $Q$ at a point} $\bar{x} \in {\rm cl}Q,$ if there exist sequences
$(t_n) \in S(0)$ and $(h_n) \in U_X(h)$
such that $\bar{x} + t_nh_n \in Q$ for all $n=1,2,\dots.$

The set of all first-order tangent (contingent) vectors to a set $Q \subset X$ at a point $\bar{x} \in {\rm cl}Q$ forms a closed cone which is denoted by $TQ(\bar{x})$ and is called {\it the (first-order) tangent (contingent) cone
to the set $Q$ at the point $\bar{x} \in {\rm cl}Q.$}

The tangent (contingent) cone $TQ(\bar{x})$ is often called \textit{the Bouligand
cone} \cite{Aub_Fr,Mord}, since in the thirties of the last century elements of this cone were considered by the French mathematician Georges Bouligand \cite{Boul30,Boul32}. In the literature the cone $TQ(\bar{x})$ had also been called  \textit{the cone of
directions admissible under constraints} \cite{Dub_Mil65}, \textit{the cone of directions admissible in the broad sense} \cite{Dem_Rub68},
\textit{the cone of possible directions} \cite{Dem_Rub90,Dem_Rub95}.

Recall, that a set $K \subset X$ is called \textit{a cone} if the following implication $(x \in K,\,\lambda > 0) \Rightarrow \lambda x \in K$ holds.

It follows directly from the definition that
\begin{equation*}\label{e1.1}
TQ(\bar{x}) = \limsup_{t \to 0,\,t>0}\frac{Q-\bar{x}}{t},
\end{equation*}
where the \textit{limsup} is the sequential Painlev\'{e}-Kuratowski upper/outer limit \cite{Kurat} (with respect to the norm topology of $X$) of the set-valued mapping $\displaystyle t \rightrightarrows \frac{Q-\bar{x}}{t}$ as $t \to 0_+$ ($t \to 0_+$ means $t \to 0,t>0$).

Furthermore, $h \in TQ(\bar{x})$ if and only if $\liminf\limits_{t \to 0,\,t>0}\displaystyle\frac{d(\bar{x}+th,Q)}{t} =0$, where $d(x,Q):=\inf\limits_{y \in Q}\|x-y\|$ is the distance from a point $x$ to a set $Q.$

In the finite-dimensional setting each sequence $(x_n) \subset Q$ that converges to a point $\bar{x} \in {\rm cl}Q$  generates one or even several first order tangent vectors to the set $Q$ at the point $\bar{x}$. Due to this fact, using the first order tangent cone as a local approximation of the set of feasible points of an optimization problem, one can derive both necessary optimality conditions and sufficient ones of first order for solutions of an optimization problem under question, the gap between which is minimal in the sense that it cannot be eliminated  by means of first order local approximations. To reduce this gap one needs to use local approximations of the second and higher order.

Local approximations of the second and, moreover, arbitrary order to sets, called \textit{variational sets}, were firstly introduced by Hoffmann and Kornstaedt in \cite{Hoff_Korn}. Somewhat later, under the name \textit{high-order tangent sets}, they were considered in the monograph \cite{Aub_Fr}. In subsequent years various high-order variational (contingent, adjacent and others) sets and based on them definitions of high-order derivatives for (set-valued) mappings with applications to high-order optimality conditions were the subject of many papers (see  \cite{Khanh_Tung18,Anh11,Anh14,Constantin,Gollan,Khanh_Tuan08,Ledz_Schaet,Li-Chen,Luu,Pavel,Stud} and references therein).  Following \cite{Gor06b}, in the present paper we refer to high order variational sets introduced by Hoffmann and Kornstaedt in \cite{Hoff_Korn} as high order proper tangent sets.

It is natural that the attention of many researchers was focused on second-order tangent sets and their applications to optimization problems \cite{Bakh_Gor09a,Bakh_Gor09b,Gor90,Gor05,Gor06a,Gor06b,Gor12,Gor-Rach,Gor_Tr14,Bon_Sh,Bigi_Castellani99,BonComSha,BCSha99,CMV,Cast,Cominetti,Kawasaki,Penot94,Penot98,Penot2000,Stud91}.
It was observed that for some sets $Q$ there can exist such sequences $(x_n) \subset Q$ converging to a point $\bar{x} \in {\rm cl}Q$ that generates none second-order proper tangent vector to $Q$ at $\bar{x}$ and, moreover, the second-order  proper tangent sets to some sets may be empty (an example of such set can be found in \cite{Gor06b,Bon_Sh,BCSha99,CMV,Penot98}). In the case when the second-order proper tangent set to a feasible set of a constrained optimization problem at a reference point is nonempty one can formulate a second-order necessary optimality condition \cite{Gor-Rach,Bigi_Castellani99}, however, corresponding sufficient condition can be obtain only under an additional second-order regularity condition \cite{Bakh_Gor09b,Bon_Sh,BCSha99,Bigi06}.
When the second-order proper tangent set is empty it gives no information on local structure of set and can't be used for analysis.

To overcome this disadvantage in \cite{CMV,Penot98,Penot2000} the second-order proper tangent set was supplemented with a new object, called the asymptotic second order tangent cone. An alternative definition of asymptotic second order tangent vectors was given in
\cite{Gor05,Gor06a}. For a number of (scalar, vector, set-valued) optimization problems this allowed, at least in finite-dimensional settings, to obtain both necessary optimality conditions and sufficient optimality conditions with a minimal gap between them \cite{Bakh_Gor09a,Bakh_Gor09b,Bon_Sh,BCSha99}. In \cite{Gor06b} and then in \cite{Penot17} it has been introduced the second order extended tangent cone (in \cite{Penot17} it was called the high-order tangent cone) including both the second-order proper set and the asymptotic second-order tangent cone as its subsets. It should be noted that the main constituents of the second-order extended tangent cone were considered in \cite{CMV}.
In Section 2 of the present paper, following the scheme of defining second-order extended tangent cone of the paper \cite{Gor06b} and the terminology adopted there, we introduce the high-order extended tangent cone  and study its properties. In Section 3 we use this local approximation for deriving high-order necessary conditions for local minimizers of constrained optimization problems.

\section{High order extended tangent vectors}

First we recall the definition of the variational contingent set of the $k$-th order that was given in \cite{Hoff_Korn}. Following the terminology of the paper \cite{Gor06b}, we replace the name of $k$-th order ($k \ge 2$) variational contingent set with that of \textit{$k$-th order proper tangent set}, and the vectors which belong to it we call \textit{$k$-th order proper tangent vectors}.

Let $Q$ be a set in a real normed space $X$, $\bar{x}\in {\rm cl}Q$, $k \in {\mathbb N},\,k \ge 2$, and let $(h_1, h_2, \ldots, h_{k-1}) \in
X^{k-1}:=\underbrace{X\times X\times \ldots \times
X}\limits_{k-1}$ be an ordered collection of vectors from $X.$

\begin{definition}\label{d2.1}(cf. \cite{Aub_Fr,Hoff_Korn})
A vector $w \in X$ is called {\it a $k$-th order proper tangent vector to a set $Q$ at a point $\bar{x} \in {\rm cl}Q$ with respect to an ordered collection of directions $(h_1, h_2, \ldots, h_{k-1}) \in X^{k-1}$}, if there exist sequences
$(t_n)\in S(0)$ and $(w_n)\in U_X(w)$ such that
$\bar{x}+t_nh_1+t^2_nh_2 + \ldots+
t_n^{k-1}h_{k-1}+t_n^{k}w_n\in Q$ for all $n = 1,2,\ldots~.$
\end{definition}

The set of all $k$-th order proper tangent vectors to a set
$Q \subset X$ at a point $\bar{x} \in {\rm cl}Q$ with respect to an ordered collection of directions $(h_1, h_2, \ldots, h_{k-1}) \in X^{k-1}$
is denoted by $T^k_{pr}Q(\bar{x},h_1,
\ldots,  h_{k-1})$ and is called \textit{$k$-th order proper tangent set to a set
$Q \subset X$ at a point $\bar{x} \in {\rm cl}Q$ with respect to an ordered collection of directions} $(h_1, h_2, \ldots, h_{k-1}) \in X^{k-1}$.

It follows from the equality
$$ %\begin{equation}\label{e4.1}
T^k_{pr}Q(\bar{x},h_1, \ldots,  h_{k-1}) = \limsup\limits_{t \to +0}\frac{Q-\bar{x}-th_1-t^2h_2 - \ldots -
t^{k-1}h_{k-1}}{t^{k}}
$$
and from properties of the sequential Painlev\'{e}-Kuratowski upper/outer limit (with respect to
the norm topology of $X$) that $T^k_{pr}Q(\bar{x},h_1, \ldots,  h_{k-1})$ is a closed set in $X$. Observe that \linebreak  $T^k_{pr}Q(\bar{x},h_1, \ldots,  h_{k-1})$ can be empty.

\begin{example}{\rm \cite{Gor06b,CMV,Penot98}.}\label{ex1}
Let $Q =\{(x_1,x_2) \in {\mathbb{R}}^2 \mid x_1 \ge 0, x_1^2 = x_2^3\}$ and $\bar{x} = (0,0)$. Then, it is not difficult to get by direct calculation that  $TQ(0) = \{(h_1,h_2) \in {\mathbb{R}}^2 \mid h_1 = 0, h_2 \ge 0\}$ and $T^2_{pr}Q(0,h) = \varnothing$ for all $h \in TQ(0), h \ne 0$.
\end{example}

More properties of the high-order proper tangent set of an arbitrary order can be found in the papers \cite{GIN,Hoff_Korn} and in the monographs \cite{Aub_Fr, KTZ}. Here these properties are not presented since they follows from the properties of the extended tangent cone of high order whose definition and properties are the purpose of the present paper.

\begin{definition}\label{d2.2}%\cite{Gor02}
A vector $w \in X$ is called {\it an extended tangent vector of $k$-th order to a set $Q$ at a point $\bar{x} \in {\rm cl}Q$ with respect to an ordered collection of directions $(h_1, h_2, \ldots, h_{k-1}) \in X^{k-1}$}, if there exist sequences
$(t_n),\,(\tau_n)\in S(0)$ and $(w_n)\in U_X(w)$ such that
$\bar{x}+t_nh_1+t^2_nh_2 + \ldots+
t_n^{k-1}h_{k-1}+t_n^{k-1}\tau_nw_n\in Q$ for all $n =1,2,\ldots~.$
\end{definition}

The set of all extended tangent vectors of $k$-th order to a set $Q$ at a point $\bar{x} \in {\rm cl}Q$ with respect to an ordered collection of directions $(h_1, h_2, \ldots, h_{k-1}) \in X^{k-1}$ is a cone, which is denoted by $T^kQ(\bar{x},h_1,
\ldots,  h_{k-1})$ and is called {\it the extended tangent cone of $k$-th order to a set $Q$ at a point $\bar{x} \in {\rm cl}Q$ with respect to an ordered collection of directions $(h_1, h_2, \ldots, h_{k-1}) \in X^{k-1}$}.

The following equalities hold:
\begin{equation}\label{e2.1}
T^kQ(\bar{x},h_1, \ldots,  h_{k-1}) = \limsup\limits_{(t,\tau) \to (+0,+0)}\frac{Q-\bar{x}-th_1-t^2h_2 - \ldots -
t^{k-1}h_{k-1}}{t^{k-1}\tau}
\end{equation}
and
$$
T^kQ(\bar{x},h_1, \ldots,  h_{k-1}) = \left\{w \in X \mid \liminf\limits_{(t,\tau) \to (+0,+0)}\frac{d(\bar{x}+th_1 + \ldots
+ t^{k-1}h_{k-1} + t^{k-1}\tau w,Q)}{t^{k-1}\tau}=0\right\}.
$$

It follows from Definition \ref{d2.2} and the equality \eqref{e2.1} that $T^kQ(\bar{x},h_1, \ldots,  h_{k-1})$ is a closed cone in $X$.

\smallskip
Clearly, that $T^k_{pr}Q(\bar{x},h_1, \ldots,  h_{k-1}) \subset T^kQ(\bar{x},h_1, \ldots,  h_{k-1})$.

\vspace{3mm}

In the next propositions we present a number of properties of the extended tangent cone of $k$-th order.

\smallskip

\begin{proposition}\label{pr1}  $$T^kQ(\bar{x},h_1, \ldots,  h_{k-1}) \neq \varnothing \
\Rightarrow $$  $$h_{k-1}\in T^{k-1}_{pr}Q(\bar{x},h_1, \ldots,  h_{k-2}),
 h_{k-2}\in T^{k-2}_{pr}Q(\bar{x},h_1, \ldots,  h_{k-3}), \ldots, h_2 \in T^2_{pr}Q(\bar{x},h_1), h_1 \in TQ(\bar{x}).$$

 When ${\rm dim}\,X < \infty$, the converse implication also is true: for an $\bar{x} \in {\rm cl}Q$ and an ordered collection of directions $(h_1,\ldots,h_{k-1}) \in X^{k-1}$ such that $h_{k-1}\in T^{k-1}_{pr}Q(\bar{x},h_1, \ldots,  h_{k-2})$ $($and, hence, such that $h_{k-2}\in T^{k-2}_{pr}Q(\bar{x},h_1, \ldots,  h_{k-3}),$ $h_{k-3}\in T^{k-2}_{pr}Q(\bar{x},h_1, \ldots,  h_{k-4}), \ldots, h_2 \in T^2_{pr}Q(\bar{x},h_1), h_1 \in TQ(\bar{x}))$, one has $T^kQ(\bar{x},h_1, \ldots,  h_{k-1}) \neq \varnothing$.
 \end{proposition}

\begin{proof}
 Let $T^kQ(\bar{x},h_1, \ldots,  h_{k-1}) \neq \varnothing$ and let
 $w \in T^kQ(\bar{x},h_1, \ldots,  h_{k-1}).$
 Then there exist sequences $(t_n),\,(\tau_n)\in S(0)$ and $(w_n)\in U_X(w)$ such that
 $x_n := \bar{x}+t_nh_1 + \ldots + t_n^{k-1}h_{k-1} + t_n^{k-1}\tau_n w_n \in Q\,\,\forall\,\,n.$ For each $s = 2,\ldots,k-1$ we define the sequence $(w'_n) = (h_{s} + t_n h_{s+1}+ \ldots + t_n^{k-s-1}h_{k-1} + + t_n^{k-s-1}\tau_nw_n).$ Since $w'_n \to h_{s}$ and
$x_n := \bar{x}+t_nh_1 + \ldots + t_n^{s-1}h_{s-1} +t_n^{s}w'_n \in Q\,\,\forall\,\,n,$ we obtain $h_{s} \in T^{s}_{pr}Q(\bar{x},h_1, \ldots,  h_{s-1})$ for $s=2,\ldots,k-1.$ The proof of the condition $h_1 \in TQ(\bar{x})$ is similar.

To prove the converse statement, consider a point $\bar{x} \in {\rm cl}Q$ and an ordered collection of directions $(h_1,\ldots,h_{k-1}) \in X^{k-1}$ such that $h_{k-1}\in T^{k-1}_{pr}Q(\bar{x},h_1, \ldots,  h_{k-2})$. Then there exist sequences $(t_n) \in S(0)$ and $(w_n) \in U(h_{k-1})$ such that $x_n := \bar{x} + t_nh_1 + $  $+t^2_nh_2 + \ldots +t_n^{k-2}h_{k-2} + t_n^{k-1}w_n \in Q\,\,\forall\,\,n.$

When the sequence $(w_n)$ is such that $w_n = h_{k-1}$ for infinitely many $n \in {\mathbb N}$, then, passing to a subsequence if necessary, we can suppose without loss of generality that $w_n=h_{k-1}$ for all $n$ and hence %\linebreak
$x_n := \bar{x} + t_nh_1 + t^2_nh_2 + \ldots +t_n^{k-2}h_{k-2} + t_n^{k-1}h_{k-1} + t_n^k\cdot 0 \in Q\,\,\forall\,\,n.$ We conclude from this that $0 \in T^{k}_{pr}Q(\bar{x},h_1, \ldots,  h_{k-1}).$ Thus, since $T^k_{pr}Q(\bar{x},h_1, \ldots,  h_{k-1}) \subset T^kQ(\bar{x},h_1, \ldots,  h_{k-1})$, we have $T^{k}Q(\bar{x},h_1, \ldots,  h_{k-1}) \ne \varnothing$ regardless of whether the space $X$ is finite-dimensional or not.

When $w_n =h_{k-1}$ only for finite many $n$ then without loss of generality we can assume that $w_n \ne h_{k-1}$ for all $n$. Consider the sequence  $w'_n := (\|w_n -h_{k-1}\|)^{-1}(w_n -h_{k-1})$. Since ${\rm dim}X < \infty$ we can choose a subsequence from the sequence $(w'_n)$ that converges to some vector $w,\,\|w\|=1.$ Without loss of generality, we can assume that the sequence $(w'_n)$ itself converges to $w.$  Setting $\tau_n = \|w_n- h_{k-1}\|$, we obtain $$x_n := \bar{x} + t_nh_1 + \ldots +t_n^{k-2}h_{k-2} + t_n^{k-1}w_n = \bar{x} + t_nh_1 + \ldots +t_n^{k-2}h_{k-2} + t_n^{k-1}h_{k-1} + t_n^{k-1}(w_n - h_{k-1}) = $$ $$=\bar{x} + t_nh_1 + \ldots +t_n^{k-2}h_{k-2} + t_n^{k-1}h_{k-1} + t_n^{k-1}\tau_nw'_n \in Q\,\,\forall\,\,n.$$
Since $(t_n),(\tau_n) \in S(0)$ and $(w'_n) \in U(w)$, we conclude that $w \in T^kQ(\bar{x},h_1, \ldots,  h_{k-1}),$ and \linebreak $w \ne 0.$ \end{proof}

\smallskip

\begin{definition}\label{d3} An ordered collection of directions $(h_1,\ldots,h_{k-2},h_{k-1}) \in X^{k-1}$ is called {\it admissible} for a set $Q$ at a point $\bar{x} \in {\rm cl} Q$, if $h_{k-1}\in T^{k-1}_{pr}Q(\bar{x},h_1, \ldots,  h_{k-2})$ (and, consequently,
$h_{k-2}\in T^{k-2}_{pr}Q(\bar{x},h_1, \ldots,  h_{k-3}), \ldots, h_2 \in T^2_{pr}Q(\bar{x},h_1), h_1 \in TQ(\bar{x})$).

An ordered collection of directions  $(h_1,h_2\ldots,h_{k-1}) \in X^{k-1}$ for which we can find such a sequence $(t_n) \in S(0)$, that  $\bar{x} + t_nh_1 + t^2_nh_2+\ldots + t_n^{k-1}h_{k-1} \in Q\,\,\forall\,\,n,$ will be called {\it polynomially admissible.}
\end{definition}

 \begin{remark}\label{r1}
 In the above proof of Proposition \ref{pr1} it is actually proved that for any ordered collection of directions $(h_1,h_2,\ldots,h_{k-1}) \in X^{k-1}$ that is polynomially admissible for $Q \subset X$ at a point $\bar{x} \in {\rm cl}\,Q$ one has $0 \in T^{k}_{pr}Q(\bar{x},h_1, \ldots,  h_{k-1})$ regardless of whether the space $X$ is finite-dimensional or not.
In the case when ${\rm dim}\,X < \infty$ for each admissible ordered collection of directions $(h_1,h_2,\ldots,h_{k-1}) \in X^{k-1}$ which is not polynomially admissible for $Q$ at a point $\bar{x} \in \textrm{cl}\,Q$ the extended tangent cone
$T^{k}Q(\bar{x},h_1, \ldots,  h_{k-1})$ is nonempty as well and, moreover, it contains a nonzero vector $w.$
\end{remark}

\smallskip

\begin{remark}\label{r2}
 It follows from Proposition \ref{pr1} that if for some a set $Q$, a point $\bar{x} \in Q$, and an ordered collection of directions  $(\bar{h}_1,\bar{h}_2,\ldots,\bar{h}_{k-1}) \in X^{k-1}$ ($k \ge 2$) the equality \linebreak $T^{k}_{pr}Q(\bar{x},\bar{h}_1, \ldots, \bar{h}_{k-1}) = \varnothing$ holds, then
$T^{m}Q(\bar{x},\bar{h}_1, \ldots, \bar{h}_{k-1}, h_k, \ldots, h_{m-1}) =\varnothing$ for all vectors $h_k, \ldots, h_{m-1} \in X$ and all $m > k$. In particular, since for the set $Q =\{(x_1,x_2) \in {\mathbb{R}}^2 \mid x_1 \ge 0, x_1^2 = x_2^3\}$, a point $\bar{x} = (0,0)$, and any nonzero tangent vector $h = (h_1,h_2) \in TQ(0)$ one has $T^2_{pr}Q(0,h) = \varnothing$ (see Example \ref{ex1}) then $T^{m}Q(0,{h},h_2, \ldots, h_{m-1}) = \varnothing$ for all $h_2, \ldots, h_{m-1} \in {\mathbb{R}}^2$ and all $m > 2$.
\end{remark}

\smallskip

\begin{proposition}\label{pr2} {\it For any $\bar{x} \in {\rm cl}Q$ the following equality holds}:
    $$T^kQ(\bar{x}, \underbrace{0, \ldots,  0}\limits_{k-1}) =
TQ(\bar{x})\quad \forall\,\, k\ge 2.$$
\end{proposition}

\begin{proof} For a $w \in T^kQ(\bar{x}, \underbrace{0, \ldots,  0}\limits_{k-1})$, there exist sequences $(t_n),\,(\tau_n) \in S(0)$ and $(w_n) \in U(w)$ such that $x_n:=\bar{x} + t_n0+t_n^2 0 + \ldots + t_n^{k-1} 0 + t_n^{k-1}\tau_n w_n \in Q\,\,\forall\,\,n.$ Setting ${t'}_n = t_n^{k-1}\tau_n$, we obtain $x_n=\bar{x} + {t'}_n w_n \in Q\,\,\forall\,\,n,$ with $({t'}_n) \in S(0)$ and $(w_n) \in U(w).$ Hence, $w \in TQ(\bar{x}).$

Conversely, let $w \in TQ(\bar{x})$ and let $(t_n) \in S(0)$ and $(w_n) \in U(w)$ be such sequences for which $x_n:=\bar{x} + t_nw_n \in Q\,\,\forall\,\,n.$
Then $x_n =\bar{x} + {t'}_n 0 + {t'}_n^2 0 + \ldots + {t'}_n^{k-1} 0 +  {t'}_n^{k} w_n \in Q\,\,\forall\,\,n,$ where ${t'}_n = \sqrt[k]{t_n} \to 0.$ Consequently, $w \in T^k_{pr}Q(\bar{x}, \underbrace{0, \ldots,  0}\limits_{k-1}) \subset T^kQ(\bar{x}, \underbrace{0, \ldots,  0}\limits_{k-1}).$ \end{proof}

\begin{proposition}\label{pr3} {\it For any $\bar{x} \in {\rm cl}Q$ and any $h \in TQ(\bar{x})$ the following equality holds}:
  \begin{equation}\label{e2.2}
  T^kQ(\bar{x}, \underbrace{0, \ldots,
0}\limits_{k-2},h) = T^2Q(\bar{x}, h)\,\forall\,\,
k \ge 2.
\end{equation}
\end{proposition}

\begin{proof} For a $w \in T^kQ(\bar{x}, \underbrace{0, \ldots,  0}\limits_{k-2}, h)$ there exist sequences $(t_n),\,(\tau_n) \in S(0)$ and $(w_n) \in U(w)$ such that $x_n:=\bar{x} + t_n0 + \ldots + t_n^{k-2} 0 + t_n^{k-1} h + t_n^{k-1}\tau_n w_n \in Q\,\,\forall\,\,n.$ Setting ${t'}_n = t_n^{k-1}$ and removing null summands we obtain $x_n=\bar{x} + {t'}_nh + {t'}_n\tau_n w_n \in Q\,\,\forall\,\,n,$ with $({t'}_n),\,(\tau_n) \in S(0)$ and $(w_n) \in U(w).$ Hence, $w \in T^2Q(\bar{x},h).$

Conversely, let $w \in T^2Q(\bar{x},h)$ and let $x_n:=\bar{x} + {t}_nh + {t}_n\tau_n w_n \in Q\,\,\forall\,\,n,$ where $({t}_n),\,(\tau_n) \in S(0)$ and $(w_n) \in U(w).$ Then $x_n=\bar{x} + {t'}_n0+{t'}_n^2 0 + \ldots + {t'}_n^{k-2} 0 + {t'}_n^{k-1} h + {t'}_n^{k-1}\tau_n w_n \in Q\,\,\forall\,\,n,$ where ${t'}_n = \sqrt[k-1]{t_n} \to 0.$ Consequently, $w \in T^kQ(\bar{x}, \underbrace{0, \ldots,  0}\limits_{k-2}, h)$. \end{proof}

\smallskip

Perhaps the equality \eqref{e2.2} explains the fact that in \cite{Penot17} the cone $T^2Q(\bar{x},h)$ is called a high order tangent cone but not the second order.
\smallskip

\begin{definition}\label{d4}
 Let $\alpha \in \overline{\mathbb R}_+:= {{\mathbb R}_{++}}\cup
\{0\}\cup\{+\infty\}$ and $k \ge 2.$
The subset $T_{\alpha}^kQ(\bar{x},h_1, \ldots,
h_{k-1})$ of $T^kQ(\bar{x},h_1, \ldots,  h_{k-1})$ consisting of such vectors $w \in X$ for which there exist sequences
$(t_n),\,(\tau_n) \in S(0)$ and $(w_n) \in U_X(w)$ such that $t_n^{-1}\tau_n \to \alpha$ as $n \to \infty$ and
$\bar{x}+t_nh_1+t^2_nh_2 + \ldots+
t_n^{k-1}h_{k-1}+t_n^{k-1}\tau_nw_n\in Q$ for all $n=1,2,\dots$ is called {\it the $\alpha$-slice of the extended tangent cone of $k$-th order  to a set $Q \subset X$ at a point $\bar{x} \in {\rm cl}Q$ with respect to an ordered collection of directions $(h_1, h_2, \ldots, h_{k-1}) \in X^{k-1}$}.

\end{definition}

Replacing, if necessary, sequences by subsequences, it is easy to verify that
$$ %\begin{equation}\label{e4.5}
T^kQ(\bar{x},h_1, \ldots,  h_{k-1}) = \bigcup_{\alpha \in
\overline{{\mathbb R}}_{+}}T_{\alpha}^kQ(\bar{x},h_1, \ldots, h_{k-1}).
$$ %\end{equation}

In the next propositions we present some properties of slices of the extended tangent cone.

\smallskip

\begin{proposition}\label{pr5a} \textit{Let} $\alpha > 0$. \textit{Then}
\begin{equation}\label{e2.3}
T^k_\alpha Q(\bar{x},h_1, \ldots,  h_{k-1}) = \alpha^{-1} T^k_{pr}Q(\bar{x},h_1, \ldots,  h_{k-1})\,\,\forall\,\,\alpha \in {\mathbb{R}}_{++}
\end{equation}
{\it and, hence},
$$T^k_1Q(\bar{x},h_1, \ldots,  h_{k-1}) = T^k_{pr}Q(\bar{x},h_1, \ldots,  h_{k-1})$$
{\it and}
$$\alpha_1T^k_{\alpha_1} Q(\bar{x},h_1, \ldots,  h_{k-1}) = \alpha_2T^k_{\alpha_2} Q(\bar{x},h_1, \ldots,  h_{k-1})\,\,\forall\,\,\alpha_1,\alpha_2 \in {\mathbb{R}}_{++}.$$
\end{proposition}
\begin{proof} Let $w \in T^k_\alpha Q(\bar{x},h_1, \ldots,  h_{k-1})$ and let sequences $(t_n),\,(\tau_n) \in S(0)$ and $(w_n) \in U_X(w)$ be such that $t_n^{-1}\tau_n \to \alpha$ as $n \to \infty$ and $x_n:=\bar{x}+t_nh_1+t^2_nh_2 + \ldots+t_n^{k-1}h_{k-1}+t_n^{k-1}\tau_nw_n\in Q$ for all $n=1,2,\dots$. Then $x_n=\bar{x}+t_nh_1+t^2_nh_2 + \ldots+t_n^{k-1}h_{k-1}+t_n^{k}w'_n\in Q$ for all $n=1,2,\dots$,
where $w'_n = t_n^{-1}\tau_n w_n \to \alpha w$ as $n \to \infty$ and, hence, $\alpha w \in T^k_{pr} Q(\bar{x},h_1, \ldots,  h_{k-1}),$ that is $w \in \alpha^{-1} T^k_{pr} Q(\bar{x},h_1, \ldots,  h_{k-1}).$

Conversely, let $\alpha \in {\mathbb{R}}_{++}$ and let $w \in T^k_{pr}Q(\bar{x},h_1, \ldots,  h_{k-1}).$ Then there exist sequences $(t_n) \in S(0)$ and $(w_n) \in U_X(w)$ such that $x_n:=\bar{x}+t_nh_1+t^2_nh_2 + \ldots+t_n^{k-1}h_{k-1}+t_n^{k}w_n\in Q$ for all $n=1,2,\dots$. Setting $\tau_n = \alpha t_n$, we obtain $t^{-1}_n\tau_n = \alpha$ and $x_n=\bar{x}+t_nh_1+t^2_nh_2 + \ldots+t_n^{k-1}h_{k-1}+t_n^{k-1}\tau_n \bar{w}_n \in Q$ for all $n=1,2,\dots$, where $\bar{w}_n = \alpha^{-1} w_n \to \alpha^{-1}w$. From this we conclude that  $\alpha^{-1}w \in T^k_\alpha Q(\bar{x},h_1, \ldots,  h_{k-1}).$ \end{proof}

\smallskip

\begin{remark}\label{r3} Since $T^k_{pr}Q(\bar{x},h_1, \ldots,  h_{k-1})$ is close for all $k \ge 2$, it follows from the equality \eqref{e2.3} that for any real $\alpha > 0$ the $\alpha$-slice $T^k_\alpha Q(\bar{x},h_1, \ldots,  h_{k-1})$ also is
a close set for all $k \ge 2$.
\end{remark}

\begin{corollary}\label{cor1} {\it For any point $\bar{x} \in {\rm cl}Q$ and any ordered collection of directions $(h_1,\ldots,h_{k-1}) \in X^{k-1}$
the following equality holds:
$$
T^kQ(\bar{x}, h_1,\ldots, h_{k-1}) = ({\rm cone}\,T^k_{pr}Q(\bar{x}, h_1,\ldots, h_{k-1})) \cup  T^k_0Q(\bar{x}, h_1,\ldots, h_{k-1}) \cup T^k_\infty Q(\bar{x}, h_1,\ldots, h_{k-1}).
$$
Here ${\rm cone}\,M := \{\lambda x \mid \lambda > 0, x \in M\}$ is the conical hull of a set $M$ (the smallest cone containing a set $M$)}.
\end{corollary}

\begin{proposition}\label{pr5} {\it The slices $T^k_0Q(\bar{x}, h_1,\ldots, h_{k-1})$ and $T^k_\infty Q(\bar{x}, h_1,\ldots, h_{k-1})$ are closed cones}.

\end{proposition}

\begin{proof} The fact that $T^k_0Q(\bar{x}, h_1,\ldots, h_{m-1})$ and $T^k_\infty Q(\bar{x}, h_1,\ldots, h_{m-1})$ are cones, is verified directly from the definitions of these slices.
The closedness follows from  the equalities
$$ %\begin{equation}\label{e4.1}
T^k_0Q(\bar{x},h_1, \ldots,  h_{k-1}) = \limsup\limits_{(t,\,\tau) \to (+0,\,+0),\, t^{-1}\tau \to 0}\frac{Q-\bar{x}-th_1-t^2h_2 - \ldots -
t^{k-1}h_{k-1}}{t^{k-1}\tau} %\eqno{(7)}
$$
and
$$ %\begin{equation}\label{e4.1}
T^k_\infty Q(\bar{x},h_1, \ldots,  h_{k-1}) = \limsup\limits_{(t,\,\tau) \to (+0,\,+0),\, t^{-1}\tau \to +\infty}\frac{Q-\bar{x}-th_1-t^2h_2 - \ldots -
t^{k-1}h_{k-1}}{t^{k-1}\tau}. %\eqno{(8)}
$$

\end{proof}

\smallskip

The cone $T^k_\infty Q(\bar{x},h_1, \ldots,  h_{k-1})$ will also called {\it the asymptotic tangent cone of $k$-th order to the set $Q$ at the point $\bar{x}$ with respect to the ordered collection of directions $(h_1,h_2,\ldots,h_{k-1}) \in X^{k-1}$}.

\smallskip

\begin{proposition}\label{pr6} { \it For any $\beta \in {\mathbb{R}}_{++}$ and any $\alpha \in \overline{\mathbb R}_+$ the following equality holds}: $$T_\alpha^kQ(\bar{x},\beta h_1,\beta ^2 h_2, \ldots, \beta^{k-1}
h_{k-1}) = T_\alpha^kQ(\bar{x},h_1, \ldots, h_{k-1}).$$

\end{proposition}

\begin{proof} For any real number $\beta \in {\mathbb{R}}_{++}$ and any sequences $(t_n),(\tau_n) \in S(0), t^{-1}_n\tau_n \to \alpha$ and $(w_n) \in U(w)$ we have
$$\bar{x} + t_nh_1 + \ldots +t_n^{k-1}h_{k-1} + t_n^{k-1}\tau_n w_n = \bar{x} + {t'}_n(\beta h_1) + \ldots +{t'}_n^{k-1}(\beta^{k-1}h_{k-1}) + {t'}_n^{k-1}{\tau'}_nw_n,$$ where ${t'}_n = \beta^{-1}t_n,\,{\tau'}_n = \beta^{k-1} \tau_n.$
Since $({t'}_n),\,({\tau'}_n) \in S(0)$ and $t^{-1}_n\tau_n={t'}^{-1}_n{\tau'}_n$ , from the above equality we see, that any vector $w \in T_\alpha^kQ(\bar{x},\beta h_1,\beta^2 h_2, \ldots, \beta^{k-1}
h_{k-1})$ also belongs to \linebreak $T_\alpha^kQ(\bar{x},h_1, \ldots, h_{k-1})$ and conversely.
\end{proof}

\smallskip

Now we give another characterization of the cone $T^k_\infty Q(\bar{x},h_1, \ldots,  h_{k-1})$. For this we need the following notions.

A ray emanating from the origin and going through a nonzero vector $w \in X$, denoted by ${\rm dir}\,w$, is called {\it a direction point}. The set of all direction points is denoted by ${\rm hzn}\,X$, and is called \cite{RW} {\it the horizon of the space $X$}.

A sequence of vectors $(w_n) \subset X$ is said \cite{RW} to converge to a direction point ${\rm dir}\,w \in {\rm hzn}\,X$ if there exists a sequence
$(\gamma_n) \in S(0)$ such that $\gamma_n w_n \to w$ as $n \to \infty$.

By $U_X({\rm dir}\,w)$ we denote the collection of all sequences $(w_n)$ from $X,$ which converge to ${\rm dir}\,w.$

\smallskip

\begin{proposition}\label{pr7}  {\it For any $w \in T^k_\infty Q(\bar{x}, h_1,\ldots, h_{m-1})$ there exist sequences $(t_n) \in S(0)$ and $(w_n) \in U({\rm dir}\,w)$ such that $\bar{x}+t_nh_1+t^2_nh_2 + \ldots+t_n^{k-1}h_{k-1}+t_n^{k}w_n\in Q$ for all $n=1,2,\dots$.

Conversely, if there exist sequences $(t_n), (\gamma_n) \in S(0)$ and $(w_n) \in U({\rm dir}\,w)$ such that $\gamma_nw_n \to w,\,t_n\gamma_n^{-1} \to 0$ and $\bar{x}+t_nh_1+t^2_nh_2 + \ldots+t_n^{k-1}h_{k-1}+t_n^{k}w_n\in Q$ for all $n=1,2,\dots$, then $w \in T^k_\infty Q(\bar{x}, h_1,\ldots, h_{m-1})$.}
\end{proposition}

\begin{proof} Let $w \in T^k_\infty Q(\bar{x}, h_1,\ldots, h_{m-1})$ and let sequences$(t_n),\,(\tau_n) \in S(0)$ and $(w'_n) \in U(w)$ be such that $t_n^{-1}\tau_n \to +\infty$ and $x_n:=\bar{x}+t_nh_1+t^2_nh_2 + \ldots+t_n^{k-1}h_{k-1}+t_n^{k-1}\tau_n w'_n\in Q$ for all $n=1,2,\dots.$ Then  $x_n=\bar{x}+t_nh_1+t^2_nh_2 + \ldots+t_n^{k-1}h_{k-1}+t_n^{k}w_n\in Q\,\,\forall\,\,n$, where $w_n = t_n^{-1}\tau_n w'_n \to {\rm dir}\,w.$

Conversely, if there exist sequences $(t_n), (\gamma_n) \in S(0)$ and $(w_n) \in U({\rm dir}\,w)$ such that $\gamma_nw_n \to w,\,t_n\gamma_n^{-1} \to 0$ and $x_n:=\bar{x}+t_nh_1+t^2_nh_2 + \ldots+t_n^{k-1}h_{k-1}+t_n^{k}w_n\in Q$ for all $n=1,2,\dots$, then $x_n=\bar{x}+t_nh_1+t^2_nh_2 + \ldots+t_n^{k-1}h_{k-1}+t_n^{k-1}\tau_n w'_n\in Q\,\,\forall\,\,n,$ where $\tau_n := t_n\gamma_n^{-1} \to 0$, $t^{-1}_n\tau_n = \gamma_n^{-1} \to +\infty$, and $w'_n := \gamma_n w_n \to w.$
\end{proof}

\smallskip

\begin{proposition}\label{pr8} $$T^kQ(\bar{x}, h_1,\ldots, h_{m-1}, \underbrace{0, \ldots,
0}\limits_{k-m}) \subset T^m_0Q(\bar{x}, h_1,\ldots, h_{m-1})\,\forall\,\,0<m<k,$$ {\it in particular}, $$T^kQ(\bar{x}, h, \underbrace{0, \ldots,
0}\limits_{k-2}) \subset T^2_0Q(\bar{x},h)\quad \forall\,\, k > 2.$$

\end{proposition}

\begin{proof} For any $w \in T^kQ(\bar{x}, h_1,\ldots, h_{m-1}, \underbrace{0, \ldots,
0}\limits_{k-m})$ we can find sequences $(t_n),\,(\tau_n) \in S(0)$ and $(w_n) \in U_X(w)$ such that $x_n:=\bar{x}+t_nh_1+t^2_nh_2 + \ldots+t_n^{m-1}h_{m-1}+t^{m}_n0 + \ldots+t_n^{k-1}0+t_n^{k-1}\tau_nw_n\in Q$ for all $n=1,2,\dots$ or, after rearranging, $x_n=\bar{x}+t_nh_1+t^2_nh_2 + \ldots+t_n^{m-1}h_{m-1}+t_n^{m-1}\tau'_n w_n\in Q$ for all $n=1,2,\dots$,
 where $\tau'_n = t_n^{k-m}\tau_n.$ Since $t_n^{-1}\tau'_n = t_n^{k-m-1}\tau_n \to 0$ as $n \to \infty$, then $w \in T^m_{0} Q(\bar{x},h_1, \ldots,  h_{m-1}).$ \end{proof}

\smallskip

\begin{proposition}\label{pr9} $$T^k_0 Q(\bar{x},h_1, \ldots,  h_{k-1}) \neq \varnothing
\Rightarrow $$ $$\Rightarrow 0\in T^k_\alpha Q(\bar{x},h_1, \ldots,
h_{k-1}) \ \forall\,\, \alpha\in {\mathbb R}_{++}.$$
\end{proposition}

\begin{proof} Let $w \in T^k_0 Q(\bar{x},h_1, \ldots,  h_{k-1})$ and let sequences $(t_n),\,(\tau_n) \in S(0)$ and $(w_n) \in U_X(w)$ be such that $t_n^{-1}\tau_n \to 0$ as $n \to \infty$ and $x_n:=\bar{x}+t_nh_1+t^2_nh_2 + \ldots+t_n^{k-1}h_{k-1}+t_n^{k-1}\tau_nw_n\in Q$ for all $n=1,2,\dots$. Then $x_n=\bar{x}+t_nh_1+t^2_nh_2 + \ldots+t_n^{k-1}h_{k-1}+t_n^{k}w'_n\in Q$ for all $n=1,2,\dots$,
where $w'_n = t_n^{-1}\tau_n w_n \to 0$ as $n \to \infty$ and, hence, $0 \in T^k_{pr} Q(\bar{x},h_1, \ldots,  h_{k-1}).$ Through Proposition \ref{pr5a} we conclude that $0 \in T^k_\alpha Q(\bar{x},h_1, \ldots,h_{k-1})$ for all $\alpha\in {\mathbb R}_{++}.$ \end{proof}

\begin{example}\label{ex2}
It was shown in Example \ref{ex1} that for the set $Q =\{(x_1,x_2) \in {\mathbb{R}}^2 \mid x_1 \ge 0, x_1^2 = x_2^3\}$ and the point $\bar{x} = (0,0) \in Q$ one has  $TQ(0) = \{(h_1,h_2) \in {\mathbb{R}}^2 \mid h_1 = 0, h_2 \ge 0\}$ and $T^2_{pr}Q(0,h) = \varnothing$ for all $h \in TQ(0), h \ne 0$. Consequently, through \eqref{e2.3} we get that $T^2_\alpha Q(0,h) = \varnothing$ for all $h \in TQ(0), h \ne 0,$ and for all $\alpha > 0$. Due to Proposition \ref{pr9} the latter equalities imply $T^2_0 Q(0,h) = \varnothing$ for all $h \in TQ(0), h \ne 0.$ At last, it is verified directly that $T^2_\infty Q(0,h) = \{(w_1,w_2) \in {\mathbb{R}}^2 \mid w_1 \ge 0\}$ for all $h \in TQ(0), h \ne 0.$
\end{example}

   \smallskip

\begin{proposition}\label{pr10}
{\it For every} $m\ge 2$ the following inclusions hold:
$$ T^k_\alpha Q(\bar{x},h_1, \ldots,  h_{k-1}) \subset
T^{(k-1)m+1}_0 Q(\bar{x}, \underbrace{ 0, \ldots, 0, h_{1}}\limits_{m},
\ldots,  \underbrace{ 0, \ldots, 0, h_{k-1}}\limits_{m})\,\,\forall\,\,\alpha \in {\mathbb{R}}_{++}\cup\{0\}$$ {\it and}
$$T^{(k-1)m+1}_{\alpha} Q(\bar{x}, \underbrace{ 0, \ldots, 0, h_{1}}\limits_{m},
\ldots,  \underbrace{ 0, \ldots, 0, h_{k-1}}\limits_{m}) \subset
 T^k_\infty Q(\bar{x},h_1, \ldots,  h_{k-1})\quad \forall\alpha \in {\mathbb R}_{++}.$$

 {\it In particular}, $$T^k_\alpha Q(\bar{x},h_1, \ldots,  h_{k-1})\ \subset \ T^{2k-1}_0 Q(\bar{x}, \underbrace{ 0, h_1, 0, h_2, \ldots, 0,
h_{k-1}}\limits_{2k-2})\,\,\forall\,\,\alpha \in {\mathbb{R}}_{++}\cup\{0\}$$ {\it and}
$$ T^{2k-1}_\alpha Q(\bar{x}, 0, h_1, 0,h_2, \ldots, 0,  h_{k-1})\ \subset \
T^{k}_\infty Q(\bar{x}, h_1,  \ldots, h_{k-1})\,\, \forall\alpha \in {\mathbb
R}_{++}.$$

\end{proposition}

\begin{proof} Let $w \in T^k_\alpha Q(\bar{x},h_1, \ldots,  h_{k-1}),$ with $\alpha \in {\mathbb{R}}_{++}\cup\{0\},$ and let sequences $(t_n),\,(\tau_n) \in S(0)$ and $(w_n) \in U_X(w)$ be such that $t_n^{-1}\tau_n \to \alpha$ and $x_n :=\bar{x} + t_n h_1 + t_n^2 h_2 + \ldots + t_n^{k-1}h_{k-1} + t_n^{k-1}\tau_n w_n \in Q\,\,\forall\,\,n.$ Setting $t'_n= \sqrt[m]{t_n},$ we obtain $$x_n =\bar{x} + t'_n 0 + \ldots + (t'_n)^{m-1} 0 + (t'_n)^m h_1 + (t'_n)^{m+1} 0 + \ldots + (t'_n)^{2m-1} 0 + (t'_n)^{2m} h_2 + \ldots +$$ $$ + (t'_n)^{(k-2)m+1} 0 + \ldots + (t'_n)^{(k-1)m-1}0 + (t'_n)^{(k-1)m}h_{k-1} + (t'_n)^{(k-1)m} \tau_n w_n \in Q\,\,\forall\,\,n.$$ Since $(t'_n)^{-1}\tau_n = t_n^{\frac{m-1}{m}}(t_n^{-1}\tau_n) \to 0,$ then $w \in T^{(k-1)m+1}_0 Q(\bar{x}, \underbrace{ 0, \ldots, 0, h_{1}}\limits_{m}, \ldots,  \underbrace{ 0, \ldots, 0, h_{k-1}}\limits_{m}).$

To prove the second inclusion we assume that $\alpha \in {\mathbb R}_{++}$ and consider an arbitrary vector $w \in T^{(k-1)m+1}_{\alpha} Q(\bar{x}, \underbrace{ 0, \ldots, 0, h_{1}}\limits_{m},
\ldots,  \underbrace{ 0, \ldots, 0, h_{k-1}}\limits_{m})$ and corresponding sequences $(t_n), (\tau_n) \in S(0)$ such that $(t_n)^{-1}\tau_n \to \alpha$ and $x_n := \bar{x} + t_n 0 + \ldots + t_n^{m-1} 0 + t_n^m h_1 + t_n^{m+1} 0 + \ldots + t_n^{2m-1} 0 + t_n^{2m} h_2 + \ldots +$$ $$ + t_n^{(k-2)m+1} 0 + \ldots + t_n^{(k-1)m-1}0 + t_n^{(k-1)m}h_{k-1} + t_n^{(k-1)m} \tau_n w_n \in Q\,\,\forall\,\,n.$ Setting $t'_n = t_n^m$ and eliminating null summands, we obtain $x_n :=\bar{x} + (t'_n) h_1 + (t'_n)^2 h_2 + \ldots + (t'_n)^{k-1}h_{k-1} + (t'_n)^{k-1}\tau_n w_n \in Q\,\,\forall\,\,n,$ with $(t'_n)^{-1}\tau_n = \displaystyle\frac{1}{t_n^{m-1}}\frac{\tau_n}{t_n} \to + \infty.$ Hence, $w \in T^k_\infty Q(\bar{x},h_1, \ldots,  h_{k-1}).$ \end{proof}

\smallskip

In conclusion of this section, we note that for the first time the definition of the high order extended tangent cone and its slices, as well as a number of their properties, were presented in the author's talk at the French-German-Polish Conference on Optimization. September 9 -- 13, 2002. Cottbus, Germany \cite{Gor02}.
The second order extended tangent cone was introduced and studied in \cite{Gor06b}.

\bigskip

\section{High-order optimality conditions in smooth constrained optimization problems}

%\setcounter{equation}{0}

%\subsection
In this section our goal is to use the high-order extended tangent cone for deriving high-order necessary conditions for minimizers of the following constrained optimization problem:
$$
 \text{minimize}\,\, f(x)\,\, \text{subject to}\,\, x \in Q,
$$
where $f: X \to {\mathbb{R}}$ is a real-valued function defined on a normed space $X$, and $Q$ is a subset of $X$.

We suppose that the function $f$ is $k$ times Fr\'{e}chet differentiable, where $k \ge 2$.

Recall the definition of Fr\'{e}chet differentiability of high order.

Let $X$ and $Y$ be real normed spaces, and let ${\mathcal{L}}(X,Y)$ be the vector space of linear continuous mapping from $X$ into $Y$ endowed with the norm $\|L\|:=\sup\limits_{\|x\|\le 1}\|Lx\|$.

A mapping $f:X \to Y$ is said to be \textit{Fr\'{e}chet differentiable at a point } $\bar{x} \in X$ if there exists a linear continuous mapping $L \in {\mathcal{L}}(X,\mathbb{R})$ such that
\begin{equation}\label{e3.1}
\lim_{h \to 0}\frac{f(\bar{x}+h)-f(\bar{x}) - Lx}{\|h\|} = 0.
\end{equation}
A linear mapping $L \in {\mathcal{L}}(X,\mathbb{R})$ satisfying \eqref{e3.1} is called \textit{the Fr\'{e}chet derivative of $f$ at $\bar{x}$} and is denoted by $f'(\bar{x}).$

If a mapping $f:X \to Y$ is Fr\'{e}chet differentiable at every point of some neighborhood $U$ of a point $\bar{x}$ we obtain the derived mapping $f': x \to f'(x)$ from $U$ into ${\mathcal{L}}(X,Y)$. In the case when the mapping $f': U \to  {\mathcal{L}}(X,Y)$ is Fr\'{e}chet differentiable at the point $\bar{x}$ the mapping $f$ is said to be \textit{twice Fr\'{e}chet differentiable at} $\bar{x}$
and the derivative $(f')'(\bar{x})$ is called \textit{the second Fr\'{e}chet derivative of $f$ at $\bar{x}$} and is denoted by $f''(\bar{x})$. Note that $f''(\bar{x})$ is a linear mapping from $X$ into ${\mathcal{L}}(X,{\mathcal{L}}(X,Y))$.

For $k > 2$ the $k$-order Fr\'{e}chet derivatives  are defined by induction:
$$
f^{(k)}(x):=(f^{(k-1)})'(x) \in \underbrace{{\mathcal{L}}(X,{\mathcal{L}}(X,\ldots {\mathcal{L}}(X,{\mathcal{L}}(X,Y)))\ldots)}_k.
$$
The $k$-order Fr\'{e}chet derivative $f^{(k)}$ exists at a point $\bar{x}$ if $f'(x),f''(x),\ldots,f^{(k-1)}(x)$ exist for every $x$ in a some neighborhood of $\bar{x}$ and $f^{(k-1)}: x \to f^{(k-1)}(x)$ is Fr\'{e}chet differentiable at $\bar{x}$.

%A mapping $f: X \to Y$ is said to be of class $C^k$ on $U$ ($U$ is an open set in $X$) if $F^{(k)}(x)$ exists at every point $x \in U$ and the mapping $U \ni x \to f^{(k)}(x) \in \underbrace{{\mathcal{L}}(X,{\mathcal{L}}(X,\ldots {\mathcal{L}}(X,{\mathcal{L}}(X,Y)))\ldots)}_k$ is continuous.

In view of the isometry of the space $\underbrace{{\mathcal{L}}(X,{\mathcal{L}}(X,\ldots {\mathcal{L}}(X,{\mathcal{L}}(X,Y)))\ldots)}_k$ with the space ${\mathcal{L}}(X^k,Y)$ of $k$-linear continuous mappings the $k$-order Fr\'{e}chet derivative $f^{(k)}(\bar{x})$ of $f$ at $\bar{x}$ is identified with the corresponding element of ${\mathcal{L}}(X^k,Y)$. Moreover, the $k$-linear mapping corresponding the $k$-order Fr\'{e}chet derivative $f^{(k)}(\bar{x})$ is symmetric.

Recall that a mapping $T:X^k \to Y$ is called $k$-linear if the mappings $$X \ni z \to  T(x_1,\ldots,x_{i-1},z,x_{i+1},\ldots, x_k) \in Y,i =1,\ldots,k,$$ are linear for any fixed $x_1,\ldots,x_{i-1},x_{i+1},\ldots, x_k \in X$. A $k$-linear mapping $T:X^k \to Y$ is symmetric if $T(x_1,\ldots,x_k)$ does not change its value for any permutation of the arguments $x_1,\ldots,x_k.$

For any $k$-linear symmetric mapping $T:X^k \to Y$ by $T[x_1]^{\alpha_1}\ldots[x_\mu]^{\alpha_\mu}$, where $\alpha_i$ are nonnegative integers such that $\alpha_1+\ldots+\alpha_\mu = k$, we denote the value of this mapping when $\alpha_1$ of its arguments are equal to $x_1$, $\alpha_2$ of its arguments are equal to $x_2$, $\ldots$, $\alpha_\mu$ of its arguments equal to $x_\mu$.

%\subsection{Optimality conditions for constrained minimizing problems}
\medskip

A point $\bar{x} \in X$ is called \textit{a local minimizer of $f$ over $Q$} if $\bar{x} \in Q$ and there exists a positive real $\delta > 0$ such that $f(\bar{x}) \le f(x)$ for all $x \in Q\cap B_\delta(\bar{x})$, where $B_\delta(\bar{x}):= \{x \in X \mid \|x - \bar{x}\| \le \delta\}.$

It is well known that if $\bar{x} \in X$ is a local minimizer of a Fr\'{e}chet differentiable function $f$ over a set $Q \subseteq X$ then $f'(\bar{x})h \ge 0$ for all $h \in TQ(\bar{x})$.

In the next theorem we present the high order necessary conditions for local minimizers of smooth functions.

{\bf Theorem 1.} {\it Let a function $f:X \to {\mathbb{R}}$ be $k$ times Fr\'{e}chet differentiable at a point $\bar{x} \in X$, where $k \ge 2$. If the point $\bar{x} \in X$ is a local minimizer of $f$ over a set $Q \subseteq X$ then for any ordered collection of directions $(h_1,h_2,\ldots,h_{k-1}) \in X^{k-1}$ such that $T^kQ(\bar{x},h_1,\ldots,h_{k-1}) \ne \varnothing$ and
%all $h_1 \in TQ(\bar{x})\cap{\rm ker}f'(\bar{x})$ and for all $h_2 \in T^2_{pr}Q(\bar{x},h_1)$, $h_3 \in T^3_{pr}Q(\bar{x},h_1,h_2)$, $\ldots$ $h_{s-1} \in T^{s-1}_{pr}Q(\bar{x},h_1,\ldots,h_{s-2})$ such that
\begin{equation}\label{e3.2}
\sum\limits_{\alpha_1+2\alpha_2+\ldots+s\alpha_s=s}\frac{1}{\alpha_1!\ldots\alpha_s!}f^{(\alpha_1+\ldots+\alpha_s)}(\bar{x})[h_1]^{\alpha_1}\ldots[h_s]^{\alpha_s} = 0, \,\, s=1,2,\ldots,k-1,
\end{equation}
one has
\begin{multline}\label{e3.3}
f'(\bar{x})w + \sum\limits_{\alpha_1+2\alpha_2+\ldots+(k-1)\alpha_{k-1}=k}\frac{1}{\alpha_1!\ldots\alpha_{k-1}!}f^{(\alpha_1+\ldots+\alpha_{k-1})}(\bar{x})[h_1]^{\alpha_1}\ldots[h_{k-1}]^{\alpha_{k-1}} \ge 0 \\ \text{for all}\,\, w \in T^k_{pr}Q(\bar{x},h_1,\ldots,h_{k-1})
\end{multline}
and
\begin{equation}\label{e3.4}
    f'(\bar{x})w \ge 0\,\,\text{for all}\,\,w \in T^k_{\infty}Q(\bar{x},h_1,\ldots,h_{k-1}).
\end{equation}
}
%Here ${\rm ker}f'(\bar{x}):=\{h \in X \mid f'(\bar{x})h = 0\}$ stands for the kernel of the linear functional $f'(\bar{x})$.}

\begin{proof} Let $w \in T^kQ(\bar{x},h_1,\ldots,h_{k-1})$. Then there exist sequences $(t_n),(\tau_n) \in S(0)$ and $(w_n) \in U(w)$ such that $x_n:=\bar{x} +t_nh_1+t^2_nh_2+\ldots+t_n^{k-1}h_{k-1}+ t_n^{k-1}\tau_nw_n \in Q\,\,\forall\,\,n \in {\mathbb{N}}$. Clearly, $x_n \to \bar{x}$ as $n \to \infty$. Hence, if $\bar{x}$ is a local minimizer of $f$ over $Q$, one has $f(x_n) - f(\bar{x}) \ge 0$ for sufficiently large $n$. Using the Taylor formula we obtain
$$f(x_n) - f(\bar{x}) = f(\bar{x} +t_nh_1+t^2_nh_2+\ldots+t_n^{k-1}h_{k-1}+ t_n^{k-1}\tau_nw_n) - f(\bar{x}) =$$ $$\sum\limits_{s=1}^{k-1}t^s\sum\limits_{\alpha_1+2\alpha_2+\ldots+s\alpha_s=s}\frac{1}{\alpha_1!\ldots\alpha_s!}f^{(\alpha_1+\ldots+\alpha_s)}(\bar{x})[h_1]^{\alpha_1}\ldots[h_s]^{\alpha_s}+$$
$$ t_n^k\left(\frac{\tau_n}{t_n}f'(\bar{x})w_n + \sum\limits_{\alpha_1+2\alpha_2+\ldots+(k-1)\alpha_{k-1}=k}\frac{1}{\alpha_1!\ldots\alpha_{k-1}!}f^{(\alpha_1+\ldots+\alpha_{k-1})}(\bar{x})[h_1]^{\alpha_1}\ldots[h_{k-1}]^{\alpha_{k-1}}\right) + $$ $$+ t_n^kv_n \ge 0,\,\,\text{where}\,\,v_n \to 0\,\,\text{as}\,\,n \to \infty.$$
Taking into account the equalities \eqref{e3.2} and dividing by $t_n^k$ we get
\begin{equation}\label{e3.5}
\frac{\tau_n}{t_n}f'(\bar{x})w_n + \sum\limits_{\alpha_1+2\alpha_2+\ldots+(k-1)\alpha_{k-1}=k}\frac{1}{\alpha_1!\ldots\alpha_{k-1}!}f^{(\alpha_1+\ldots+\alpha_{k-1})}(\bar{x})[h_1]^{\alpha_1}\ldots[h_{k-1}]^{\alpha_{k-1}} + v_n \ge 0.
\end{equation}
Assume that $w \in T^k_{pr}Q(\bar{x},h_1,\ldots,h_{k-1})$. Then $t_n^{-1}\tau_n \to 1$ as $n \to \infty$ and we come from the inequality \eqref{e3.5} to \eqref{e3.3}.

If $w \in T^k_{\infty}Q(\bar{x},h_1,\ldots,h_{k-1})$ then $t_n^{-1}\tau_n \to +\infty$ and we get from \eqref{e3.5} that $f''(\bar{x})w \ge 0.$ This proves the inequality \eqref{e3.4}.

Before completing the proof, we note that the cases when $w \in T^k_{\alpha}Q(\bar{x},h_1,\ldots,h_{k-1})$ with $\alpha > 0,\alpha \ne 1,$ and when $w \in T^k_0Q(\bar{x},h_1,\ldots,h_{k-1})$ produce the conditions that are already included in \eqref{e3.4}. Really, it follows from Proposition \ref{pr5a} that $w \in T^k_{\alpha}Q(\bar{x},h_1,\ldots,h_{k-1})\,\,\Rightarrow\,\, \alpha w \in T^k_{pr}Q(\bar{x},h_1,\ldots,h_{k-1})$. In the second case, due to Proposition \ref{pr9} we have $T^k_0Q(\bar{x},h_1,\ldots,h_{k-1}) \ne \varnothing\,\,\Rightarrow \,\,0 \in T^k_{pr}Q(\bar{x},h_1,\ldots,h_{k-1}).$
\end{proof}

\begin{corollary}{\rm (cf. \cite{CMV,Penot94,Penot98})}
 Let a function $f:X \to {\mathbb{R}}$ be twice Fr\'{e}chet differentiable at a point $\bar{x} \in X$. If the point $\bar{x}$ is a local minimizer of $f$ over a set $Q \subseteq X$ then for any $h \in X$ such that $T^2Q(\bar{x},h) \ne \varnothing$ and $f'(\bar{x})h = 0$ one has
\begin{equation}\label{e3.6}
f'(\bar{x})w +\frac{1}{2}f''(\bar{x})[h]^2 \ge 0\,\,\forall\,\,w \in T^2_{pr}Q(\bar{x},h)
\end{equation}
and
\begin{equation}\label{e3.7}
f'(\bar{x})w \ge 0 \,\,\forall\,\,w \in T^2_{\infty}Q(\bar{x},h).
\end{equation}

\end{corollary}

\begin{corollary}{\rm (cf. \cite{Constantin})}
Let a function $f:X \to {\mathbb{R}}$ be three times Fr\'{e}chet differentiable at a point $\bar{x} \in X$. If the point $\bar{x}$ is a local minimizer of $f$ over a set $Q \subseteq X$ then for any $h_1,h_2 \in X$ such that $T^3Q(\bar{x},h_1,h_2) \ne \varnothing$ and $f'(\bar{x})h_1 = 0,$ $f'(\bar{x})h_2 +\frac{1}{2}f''(\bar{x})[h_1]^2 = 0$ one has
$$
f'(\bar{x})w + f''(\bar{x})[h_1][h_2] + \frac{1}{3!}f'''(\bar{x})[h_1]^3 \ge 0\,\,\forall\,\,w \in T^3_{pr}Q(\bar{x},h_1,h_2)
$$
and
$$
f'(\bar{x})w \ge 0 \,\,\forall\,\,w \in T^3_{\infty}Q(\bar{x},h_1,h_2).
$$
\end{corollary}

\begin{example}
Consider the problem of minimizing the function $f(x_1,x_2) = -x_1+x_2^3$ over the set $Q =\{(x_1,x_2) \in {\mathbb{R}}^2 \mid x_1 \ge 0, x_1^2 = x_2^3\}$. As it was shown in Examples \ref{ex1} and \ref{ex2} for the feasible point $\bar{x} = (0,0) \in Q$ we have $TQ(0) = \{(h_1,h_2) \in {\mathbb{R}}^2 \mid h_1 = 0, h_2 \ge 0\}$, $T^2_{pr}Q(0,h) = \varnothing$ and $T^2_\infty Q(0,h) = \{(w_1,w_2) \in {\mathbb{R}}^2 \mid w_1 \ge 0\}$ for all $h \in TQ(0), h \ne 0$. Since $f'(0)= (-1,0)$ then $f'(0)h = 0$ for all $h \in TQ(0)$. The condition \eqref{e3.6} holds trivially because $T^2_{pr}Q(0,h) = \varnothing$ for all $h \in TQ(0), h \ne 0$. However, $f'(0)w = -w_1 < 0$ for all $ w \in T^2_\infty Q(0,h), w \ne 0$, and, consequently, the condition \eqref{e3.7} does not hold. Thus, the point $\bar{x} = (0,0) \in Q$ is not a local minimizer of $f$ over $Q$.
\end{example}

\vskip 6mm
\noindent{\bf Acknowledgements}

\noindent
The author was supported by the National Program for Scientific Research of
the Republic of Belarus for 2021--2025 “Convergence--2025”, project No.~1.3.01.

\end{document}